\documentclass[]{article}

\usepackage{float}
\usepackage[utf8]{inputenc}

\usepackage[dvipsnames]{xcolor}
\usepackage{tikz}
\usepackage{multirow}
\usepackage{soul}
\usepackage[symbol]{footmisc}

\usepackage{nicefrac}

\tikzstyle{svertex}=[circle,inner sep=0.cm, minimum size=1.3mm, fill=black, draw=black]
\tikzstyle{novertex}=[rectangle]

\tikzset{font={\fontsize{8pt}{12}\selectfont}}

\definecolor{myRed}{RGB}{255,51,51}
\definecolor{myDefGreen}{RGB}{0,190,76}

\usepackage{enumerate}
\usepackage{todonotes}
\usepackage{mathtools}

\usepackage{amsmath, amssymb, amsfonts, amsthm}
\usepackage{pifont}

\newtheorem{theorem}{Theorem}
\newtheorem*{theorem*}{Theorem}
\newtheorem{corollary}[theorem]{Corollary}
\newtheorem{lemma}[theorem]{Lemma}
\theoremstyle{definition}

\newcommand{\xmark}{\ding{55}}
\newcommand{\cmark}{\ding{51}}

\title{$2.5$-Connectivity: Unique Components, Critical Graphs, and Applications}
\author{Irene Heinrich\footnote{Irene Heinrich has received funding from the
		European Research Council (ERC) under the European Union’s Horizon 2020
		research and innovation programme (EngageS: grant agreement No.\ 820148).}, Till Heller, Eva Schmidt\footnote{Eva Schmidt was partially funded by the Federal Ministry of Education and Research (BMBF) of Germany.}, and Manuel Streicher\footnote{Manuel Streicher was partially funded by the Federal Ministry of Education and Research (BMBF) of Germany.}}

\begin{document}
\maketitle

\begin{abstract}
	If a biconnected graph stays connected after the removal of an arbitrary vertex and an arbitrary edge, then it is called 2.5-connected.
	We prove that every biconnected graph has a canonical decomposition into 2.5-connected components. These components are arranged in a tree-structure.
	We also discuss the connection between 2.5-connected components and triconnected components and use this to present a linear time algorithm which computes the 2.5-connected components of a graph.
	We show that every critical 2.5-connected graph other than $K_4$ can be obtained from critical 2.5-connected graphs of smaller order using simple graph operations. 
	Furthermore, we demonstrate applications of 2.5-connected components in the context of cycle decompositions and cycle packings.
\end{abstract}

\section{Introduction}
Over the years, connectivity has become an indispensable notion of graph theory.
A tremendous amount of proofs start with a reduction which says
``The main result holds for all graphs if it holds for all \emph{sufficiently connected} graphs.''
Here, \emph{sufficiently connected} stands for some measure of connectedness as, for example, 
biconnected, 4-edge-connected, vertex-edge-connected, or just connected. Usually, first a reduction from the desired statement for general graphs to sufficiently connected graphs is proven. Then the subsequent section starts with a sentence of the following manner: ``From now on, all considered graphs are sufficiently connected."
For example, it is shown in~\cite{tutte1954} that the Tutte polynomial is multiplicative over the biconnected components of a considered graph. Another reduction to components of higher connectivity is that finding a planar embedding can be reduced to embedding the triconnected components of the graph, cf.~\cite{maclane1937}.
\begin{table}[htbp!]
	\centering
	\footnotesize
	\begin{tabular}{c|cccc} \label{tab: kindsOfConnectivity}
		& connected after & unique & tree-structured & example of \\
		&  removal of  & components & components & an application \\ 
		\hline 
		connected & --- & \cmark & \xmark & various \\
		& \phantom{\small{.}} & & & \\
		biconnected &  1 vertex& \cmark & \cmark & Tutte polynomial \\
		& \phantom{\small{.}} & & &\\
		\textcolor{ForestGreen}{2.5-connected} &  \textcolor{ForestGreen}{1 edge + 1 vertex}& \textcolor{ForestGreen}{\cmark} & \textcolor{ForestGreen}{\cmark} & cycle decomposition\\
		& \phantom{\small{.}} & & & \\
		triconnected & 2 vertices& \cmark & \cmark & planarity test \\
	\end{tabular}
\caption{2.5-connectivity as an intermediate connectivity measure.}
	\label{tab: conjecture summary}
\end{table}

By far the largest part of the existing literature treats either \emph{$k$-connectivity} (where, loosely speaking, graphs which stay connected even if $k-1$ vertices are removed are considered) or \emph{$k$-edge connectivity} (where graphs which stay connected even if $k-1$ edges are removed are considered).
We speak of \emph{mixed connectivity}, when graphs are regarded which stay connected after $k$ vertices and $l$ edges are removed.
This measure of connectivity has only rarely been studied.
We refer the reader to~\cite{beineke2012} for a brief survey of mixed connectivity.
In~\cite{heinrichStreicher2019} it is shown that the behaviour of cycle decompositions is preserved under splits at vertex-edge separators (that is, a vertex and an edge whose removal disconnects the graph).

\paragraph{Our contribution}
We introduce a canonical decomposition of a graph into its 2.5-connected components, where a graph is 2.5-connected if it is biconnected and the removal of a vertex and an edge does not disconnect the graph.
We prove the following decomposition theorem.

\begin{theorem}[Decomposition into 2.5-connected components]
	Let $G$ be a bicon\-nec\-ted graph.
	The 2.5-connected components of $G$ are unique and can be computed in linear time.
\end{theorem}

Furthermore, we demonstrate that the behaviour of critical 2.5-connected graphs is preserved in their triconnected components.
We obtain a result similar to Tutte's decomposition theorem for 3-connected 3-regular graphs: all critical 2.5-connected graphs other than $K_4$ can be obtained from critical 2.5-connected graphs of smaller order by simple graph operations.

Finally, we show that the minimum (maximum) cardinality of a cycle decomposition of an Eulerian graph can be obtained from the minimum (maximum) cardinalities of the cycle decompositions of its 2.5-connected component.
This gives new insights into a long standing conjecture of Hajós.

\paragraph{Techniques}
We demonstrate that 2.5-connected components can be defined in the same manner as triconnected components. The novel underlying idea of the present article is a red-green-colouring of the virtual edges of the triconnected components (a virtual edge of a component is not part of the original graph but stores the information where the components need to be glued together in order to obtain the host graph).
The colouring is assigned to the virtual edges during the process of carrying out splits that give the triconnected components.
It preserves the information whether a virtual edge could arise in a sequence of 2.5-splits (those corresponding to a vertex-edge separator). If so, the edge is coloured green, otherwise red.
We prove that this colouring can be assigned to the virtual edges of the triconnected components (without knowledge of the splits that led there) in linear time.
It can be exploited to obtain 2.5-connected components: glue the red edges.
We show that the uniqueness of the red-green-colouring implies the uniqueness of the 2.5-connected components. 

\paragraph{Further Related Work}
We refer to~\cite{tutte1966} as a standard book on graph connectivity.
The same topic is considered from an algorithmic point of view in~\cite{nagamochi2008}.
A short overview on mixed connectivity with strong emphasis on partly raising Menger's theorem to mixed separators can be found in Chapter 1.4 of~\cite{beineke2012}.
Grohe~\cite{grohe2016} introduces a new decomposition of a graph into quasi-4-connected components and discusses the relation of the quasi-4-connected components to triconnected components.

The importance of triconnected components for planarity testing was already observed in~\cite{maclane1937}.
Hopcroft and Tarjan~\cite{hopcroftTarjan1973} proved that these components are tree-structured and exploited this algorithmically.
On this basis, Battista and Tamassia~\cite{battista1996} developed the notion of SPQR-trees.
Gutwenger and Mutzel~\cite{gutwengerMutzel2000} used this result and the results of~\cite{hopcroftTarjan1973} for a linear-time algorithm that computes the triconnected components of a given graph and their tree-structure (SPQR-tree).

\paragraph{Outline}
Preliminary results and definitions are introduced in the next section.
In particular, Hopcroft and Tarjan's notions of triconnected components and virtual edges (cf.~\cite{hopcroftTarjan1973}) are explained.
In Section~\ref{sec: 25Connectivity} we adapt the definition of triconnected components in order to give a natural definition of 2.5-connected components.
We prove that these are unique and show how they can be obtained from the triconnected components. We exploit this knowledge in Section~\ref{sec: linearTime} in order to give a linear time algorithm which computes the 2.5-connected components of a given graph.
We characterize the critical 2.5-connected graphs in Section~\ref{sec: CriticalGraps}. Finally, some applications of  2.5-connected graphs are discussed in Section~\ref{sec: Applications}.

\section{Preliminaries} \label{sec: preliminaries}
If not stated otherwise, we use standard graph theoretic notation as can be found in~\cite{diestel2000}.
Graphs are finite and may contain multiple edges but no loops.
A graph of order~2 and size $k \geq 2$ is a \emph{multiedge} (or \emph{$k$-edge}).
In this article a graph $G$ is equipped with an injective labelling $\ell_G \coloneqq E_V \to \mathbb{N}$ where $E_V$ is a (possibly empty) subset of $E(G)$.
We call $E_V(G) \coloneqq E_V$ the \emph{virtual edges} of $G$.
If $G$ is described without a labelling, then we implicitly assume $E_V(G) = \emptyset$.

Most of the notation and all of the results in this paragraph are borrowed from~\cite{hopcroftTarjan1973}.
A connected graph is \emph{biconnected} if for each triple of distinct vertices $(u,v,w) \in V(G)^3$ there exists a $u$-$v$-path $P$ in $G$ with $w \notin V(P)$.\footnote{This differs from the definition of 2-connected graphs as can be found in~\cite{diestel2000}. Connected graphs of order~2 are biconnected but not 2-connected.}

Let $u$ and $v$ be two vertices of a biconnected graph~$G$.
We divide $E(G)$ into equivalence classes $E_1, E_2, \dots, E_k$ such that two edges lie in the same class if and only if they are edges of a (possibly closed) subpath of~$G$ which neither contains $u$ nor $v$ internally.
The classes~$E_i$ are the \emph{separation classes} of $G$ with respect to $\{u,v\}$.
The set $\{u,v\}$ is a \emph{separation pair} if there exists a set $I \subsetneq \{1, \dots, k\}$ such that $E' \coloneqq \bigcup_{i \in I}E_i$ satisfies $\min\{|E'|,|E(G) \setminus E'|\} \geq 2$.
In this case, let $G_1 \coloneqq G[E'] +e_1$ and $G_2 \coloneqq G[E(G)\setminus E'] + e_2$, where both, $e_1$ and $e_2$, are new edges with endvertices $u$ and $v$.
Fix some $x \in \mathbb{N}\setminus l_G(E_V(G))$.
For $i \in \{1,2\}$ let $\ell_{G_i}\colon \left(E_V(G)\cap E(G_i)\right) \cup \{e_i\} \to \mathbb{N}$ be the labelling with $\ell_{G_i}(e) = \ell_G(e)$ for $e \in E_V(G)\cap E(G_i)$ and $\ell_{G_i}(e_i) = x$.
Replacing $G$ by $G_1$ and $G_2$ is a \emph{split}.
The virtual edges $e_1$ and $e_2$ \emph{correspond} to each other.
Vice versa, if $G_1$ and $G_2$ can be obtained by a split from $G$, then $G$ is the \emph{merge graph} of $G_1$ and $G_2$.
Replacing $G_1$ and $G_2$ by $G$ is a \emph{merge}.
A biconnected graph without a separation pair is \emph{triconnected}.

Suppose a multigraph $G$ is split, the split graphs are split, and so on, until no more splits are possible. (Each graph remaining is triconnected). The graphs constructed this way are called \emph{split components of $G$}.

We say that two graphs $H$ and $H'$ are \emph{equivalent}, if $H'$ can be obtained from $H$ by renaming and relabelling the virtual edges in $E_V(H)$.
Two sets of graphs $\{G_1, \dots, G_k\}$ and $\{G_1', \dots, G_k'\}$ are \emph{equivalent} if the elements can be ordered in such a way that $G_i$ is equivalent to $G_i'$ for all $i \in \{1, \dots, k\}$ and the correspondence of the virtual edges is preserved by the according renaming and relabelling maps.
Two sets of split components of the same graph are not equivalent in general. Consider for example a cycle of length~4. The two possible separation pairs yield different partitions of the edge set of the cycle.

Split components of $G$ are of one of the following types:
\[\text{triangles,\quad 3-edges, \quad and other triconnected graphs.}\]
Denote the latter set by $\mathcal{T}$.
Merge the triangles of the split components as much as possible to obtain a set of cycles~$\mathcal{C}$.
Further, merge the 3-edges as much as possible to obtain a set of multiedges $\mathcal{M}$.
The set $\mathcal{C} \cup \mathcal{M} \cup \mathcal{T}$ is the set of \emph{triconnected components} of~$G$.
Indeed, it is accurate to speak of \emph{the} triconnected components as the following statement of Hopcroft and Tarjan~\cite{hopcroftTarjan1973} shows:

\begin{theorem}[Uniqueness of triconnected components \cite{hopcroftTarjan1973}] \label{thm: triconnCompUnique}
	If $\mathcal{I}$ and $\mathcal{I}$' are two sets of triconnected components of the same biconnected graph, then $\mathcal{I}$ and $\mathcal{I}'$ are equivalent.
\end{theorem}

The following statement is crucial for the proof of Theorem~\ref{thm: triconnCompUnique}, cf.~\cite{hopcroftTarjan1973}. We will discuss in the next section how a variation of Lemma~\ref{lem: splitsSuffice} serves us in proving the uniqueness of 2.5-connected components.
\begin{lemma}[\cite{hopcroftTarjan1973}]
	\label{lem: splitsSuffice}
	Let $\mathcal{I}$ be a set of graphs obtained from a biconnected graph $G$ by a sequence of splits and merges.
	\begin{enumerate}[(a)]
		\item The graph $S(\mathcal{I})$ with
		\begin{align*}
		V(S(\mathcal{I})) = \mathcal{I},~E(S(\mathcal{I})) =\{st\colon \text{$s$ and $t$ contain corresponding virtual edges}\}
		\end{align*}
		is a tree.
		\item 
		The set $\mathcal{I}$ can be produced by a sequence of splits.
	\end{enumerate}
	
\end{lemma}

\section{2.5-Connectivity} \label{sec: 25Connectivity}
In the following, we transfer the above notation of~\cite{hopcroftTarjan1973} to \emph{mixed connectivity}, where separators may contain both, vertices and edges.
Given a biconnected graph which is not a
triangle, a tuple $(c, uv) \in V(G) \times E(G)$ is a \emph{vertex-edge-separator} if $G-uv-c$ is disconnected.

\begin{lemma}[\cite{heinrichStreicher2019}]
	\label{lem: veSeparationUniqueAndBiconnected}
	Let $G$ be a biconnected graph.
	If $(c, uv)$ is a vertex-edge-separator of~$G$, then $G-uv-c$ has exactly two components, each containing a vertex of $\{u,v\}$.
	Let $a \in \{u,v\}$.
	Denote the component containing $a$ by $C_a$ and set $G_a \coloneqq G[V(C_a) \cup \{c\}]$.
	Then $G_a + ca$ is biconnected.
\end{lemma}

With the same notation as in Lemma~\ref{lem: veSeparationUniqueAndBiconnected}, 
it holds $\max_{a \in \{u,v\}}\{|E(G_a)|\} \geq 2$. Let $a \in \{u,v\}$.
If $|E(G_a)| \geq 2$, then $\{c,a\}$ is a separation pair of~$G$.
Let~$b$ denote the vertex in $\{u,v\}\setminus\{a\}$.
Now $G_a+ac$, $G_b+ba+ac$ are split graphs of $\{c,a\}$ with virtual edges $ac$.
We say that $\{a,c\}$ \emph{supports} the vertex-edge-separator $(c, uv)$ or that $\{a,c\}$ is \emph{supporting}. Replacing $G$ by the two graphs $G_a + ac$ and $G_b + ba + ac$ is called \emph{2.5-split} of~$G$ at $(c, uv)$ \emph{with support} $\{a,c\}$.
The graphs $G_a+ac$ and $G_b + ba + ac$ are \emph{the 2.5-split graphs} of $G$ at $(c, uv)$ with support $\{a,c\}$.
A \emph{non-supporting} split is a split which is not of this form for any vertex-edge separator.
Observe that a vertex-edge-separator has at least one and at most two separation pairs in its support.

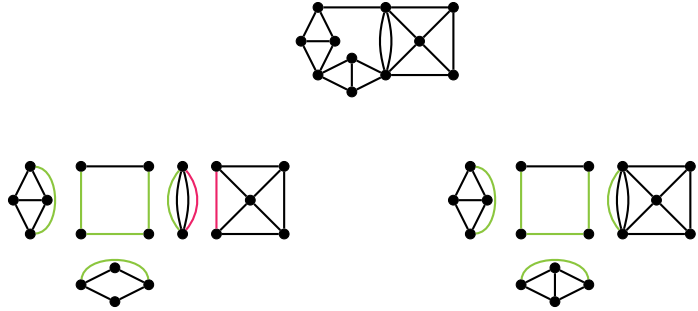
\begin{figure}[h]				
	\centering		
	\begin{tikzpicture}[scale=0.5]
	\begin{scope}[scale=.9,shift={(2,4.2)}]
	\draw 	
	(-2,0) node (1) [svertex] {}
	(-1,1) node (2) [svertex] {}
	(-1,-1) node (3) [svertex] {}
	(-3,-1) node (4) [svertex] {}
	(-3,1) node (5) [svertex] {}
	(-4,-0.5) node (6) [svertex] {}
	(-4,-1.5) node (7) [svertex] {}
	(-4.5,0) node (8) [svertex] {}
	(-5.5,0) node (9) [svertex] {}
	(-5,1) node (10) [svertex] {}
	(-5,-1) node (11) [svertex] {};
	\draw[thick]
	(1) to (2)
	(1) to (3)
	(1) to (4)
	(1) to (5)
	(2) to (3)
	(2) to (5)
	(3) to (4)
	(4) to (6)
	(4) to (7)
	(6) to (7)
	(11) to (6)
	(11) to (7)
	(5) to (10)
	(9) to (10)
	(8) to (10)
	(8) to (9)
	(8) to (11)
	(9) to (11);
	\draw[thick, bend left=15]
	(4) to (5);
	\draw[thick, bend right=15]
	(4) to (5);
	\end{scope}	
	
	\begin{scope}[scale=.9, shift={(-13,-4)}]
	\draw 	
	(4,0.5) node (12) [svertex] {}
	(4,1.5) node (13) [svertex] {}
	(3,1) node (14) [svertex] {}
	(5,1) node (15) [svertex] {}
	(3,2.5) node (16) [svertex] {}
	(5,2.5) node (17) [svertex] {}
	(3,4.5) node (18) [svertex] {}
	(5,4.5) node (19) [svertex] {}
	(1.5,2.5) node (20) [svertex] {}
	(1,3.5) node (21) [svertex] {}
	(2,3.5) node (22) [svertex] {}
	(1.5,4.5) node (23) [svertex] {}
	(6,2.5) node (24) [svertex] {}
	(6,4.5) node (25) [svertex] {}
	(7,2.5) node (26) [svertex] {}
	(7,4.5) node (27) [svertex] {}
	(9,2.5) node (28) [svertex] {}
	(9,4.5) node (29) [svertex] {}
	(8,3.5) node (30) [svertex] {};
	\draw[thick]
	(20) to (21)
	(20) to (22)
	(21) to (22)
	(23) to (22)
	(23) to (21)
	(18) to (19)		
	(14) to (12)
	(14) to (13)
	(15) to (13)
	(15) to (12)
	(27) to (29)
	(28) to (29)
	(28) to (26)
	(30) to (26)
	(30) to (27)
	(30) to (29)
	(30) to (28);
	\draw[thick, LimeGreen]
	(16) to (18)
	(16) to (17)
	(17) to (19);
	\draw[thick, WildStrawberry]
	(26) to (27);
	\draw[thick, bend right = 85, LimeGreen]
	(20) to (23);
	\draw[thick, bend left = 85, LimeGreen]
	(14) to (15);
	\draw[thick, bend left = 15]
	(25) to (24);
	\draw[thick, bend left = 40, WildStrawberry]
	(25) to (24);
	\draw[thick, bend right = 15]
	(25) to (24);
	\draw[thick, bend right = 40, LimeGreen]
	(25) to (24);
	\end{scope}
	\begin{scope}[scale=.9, shift={(0,1)}]
	\draw 	
	(4,-4.5) node (31) [svertex] {}
	(4,-3.5) node (32) [svertex] {}
	(3,-4) node (33) [svertex] {}
	(5,-4) node (34) [svertex] {}
	
	(3,-2.5) node (35) [svertex] {}
	(5,-2.5) node (36) [svertex] {}
	(3,-0.5) node (37) [svertex] {}
	(5,-0.5) node (38) [svertex] {}
	
	(1.5,-2.5) node (39) [svertex] {}
	(1,-1.5) node (40) [svertex] {}
	(2,-1.5) node (41) [svertex] {}
	(1.5,-0.5) node (42) [svertex] {}
	
	(6,-2.5) node (43) [svertex] {}
	(6,-0.5) node (44) [svertex] {}
	(8,-2.5) node (45) [svertex] {}
	(8,-0.5) node (46) [svertex] {}
	(7,-1.5) node (47) [svertex] {};
	\draw[thick]
	(39) to (40)
	(39) to (41)
	(40) to (41)
	(40) to (42)
	(41) to (42)
	(37) to (38)
	(33) to (32)
	(33) to (31)
	(32) to (34)
	(31) to (34)
	(32) to (31)
	(43) to (45)
	(46) to (44)
	(46) to (45)
	(46) to (47)
	(43) to (47)
	(44) to (47)
	(45) to (47);
	\draw[thick, bend right = 85, LimeGreen]
	(39) to (42);
	\draw[thick, bend left = 85, LimeGreen]
	(33) to (34);
	\draw[thick, bend left = 15]
	(44) to (43);
	\draw[thick, bend right = 40, LimeGreen]
	(44) to (43);
	\draw[thick, bend right = 15]
	(44) to (43);
	\draw[thick, LimeGreen]
	(35) to (37)
	(35) to (36)
	(38) to (36);
	\end{scope}
	\end{tikzpicture}		
	\caption{A graph with its triconnected (left) and 2.5-connected (right) components. Non-virtual edges are black. Virtual edges are green if they can be obtained by a sequence of 2.5-splits and red otherwise.}
	\label{fig: triconVs25Conn}	
\end{figure}

If $G$ is biconnected and no tuple $(v,e) \in V(G) \times E(G)$ is a vertex-edge-separator, then~$G$ is \emph{2.5-connected}.
In analogy to the notion of \emph{triconnected components} of Hopcroft and Tarjan~\cite{hopcroftTarjan1973}, we define 2.5-connected components of~$G$, see also Figure~\ref{fig: triconVs25Conn}.

Suppose a 2.5-split is carried out on $G$, the 2.5-split graphs are split by 2.5-splits, and so on, until no more 2.5-splits are possible. (Each graph remaining is 2.5-connected). The graphs constructed this way are called \emph{2.5-split components of $G$}.

Observe that 2.5-split components of a biconnected graph are not unique (a cycle with more than 3-edges serves again as an example).

The 2.5-split components of a given graph are of the following types:
\[\text{triangles, \quad multiedges of size at least~3,\quad and other 2.5-connected graphs.} \]
Let $G$ be a graph.
Consider a decomposition of $G$ into 2.5-split components, where $\mathcal{T}$ denotes the subset of triangles, $\mathcal{M}$ the set of multiedges and~$\mathcal{H}$ denotes the set of other 2.5-connected graphs in the decomposition.
Now merge the triangles in $\mathcal{T}$ as much as possible and leave the multiedges and 2.5-connected graphs unchanged.
Replace $\mathcal{T}$ in the split components by the set $\mathcal{C}$ of cycles obtained this way.
The components $\mathcal{C} \cup \mathcal{M} \cup \mathcal{H}$ obtained this way are the \emph{2.5-connected components} of $G$.
\begin{lemma}
	\label{lem: auxiliaryGraphIsTree}
	Let $\mathcal{I}$ be a set of graphs obtained from a biconnected graph $G$ by a sequence of 2.5-splits and merges.
	\begin{enumerate}[(a)]
		\item \label{itm: tree} The graph $S(\mathcal{I})$ is a tree.\footnote{Recall the definition of $S(\mathcal{I})$ from Lemma~\ref{lem: splitsSuffice}.}
		\item \label{itm: splitsSuffice} The set $\mathcal{I}$ can be produced by a sequence of splits.
		\item \label{itm: splitsMergesOnlyBy25Splits}  If $\mathcal{I}$ is a set of $2.5$-connected components of $G$, then $\mathcal{I}$ can be produced by a sequence of $2.5$-splits.
	\end{enumerate}
\end{lemma}

\begin{proof}[Proof sketch]
	Applying the exact same arguments as in the proof of Lemma~\ref{lem: splitsSuffice} in~\cite{hopcroftTarjan1973} we obtain~\eqref{itm: tree} and~\eqref{itm: splitsSuffice}.
	
	We prove~\eqref{itm: splitsMergesOnlyBy25Splits} by induction on the order of $G$.
	If $|V(G)| \leq 2$, then the claim is trivially satisfied since no 2.5-split is applicable to $G$.
	Therefore, let $|V(G)| \geq 3$.
	Let $G_l$ be a leaf of $T(\mathcal{I})$ and set $\mathcal{I}' \coloneqq \mathcal{I}\setminus \{G_l\}$.
	Denote the graph obtained from merging all graphs in $\mathcal{I}'$ at corresponding virtual edges by $G'$.
	Observe that $\mathcal{I}'$ is a set of 2.5-connected components of $G'$.
	
	By induction $\mathcal{I}'$ can be obtained by a sequence of 2.5-splits $s_1, s_2, \dots, s_k$ from $G'$.
	Let $G_n$ be the graph in $\mathcal{I}'$ that is adjacent to $G_l$ in $T(\mathcal{I})$ and let $G_{nl}$ be the graph obtained from merging $G_l$ with $G_n$.
	By the above considerations, the set $\left(\mathcal{I}'\setminus \{G_l\}\right) \cup \{G_{nl}\}$ can be obtained from $G$ by a sequence of 2.5-splits.
	Since the split of $G_{nl}$ into $G_n$ and $G_l$ is a 2.5-split, we obtain that $\mathcal{I}$ can be obtained from $G$ by a sequence of 2.5-splits.	
\end{proof}

\begin{lemma}[cf.~\cite{heinrichStreicher2019}] \label{lem: noGreenEdgesAfter25}
	Let $G_1$ and $G_2$ be split graphs of a biconnected graph $G$ with respect to some separation pair.
	If $(c,e)$ is a vertex-edge-separator of $G_1$ and $e \in E(G)$, then $(c,e)$ is a vertex-edge-separator of $G$.
\end{lemma}

\begin{lemma}
	\label{lem: triconnectedFrom25Connected} 
	Let $\mathcal{I}$ be a set of 2.5-connected components of a biconnected graph $G$.
	The triconnected components of $G$ can be obtained from $\mathcal{I}$ by a sequence of splits.
\end{lemma}

\begin{proof}
	By Lemma~\ref{lem: auxiliaryGraphIsTree}\eqref{itm: splitsMergesOnlyBy25Splits} there is a sequence $s^\prime$ of
	$2.5$-splits such that $\mathcal{I}$ can be obtained from $G$ with $s'$.
	Apply a sequence of splits $s$ to the graphs in
	$\mathcal{I}$ to obtain a set of split components $\mathcal{I}'$ of $G$. In
	order to obtain the triconnected components of~$G$ from~$\mathcal{I}'$, cycles
	(respectively multiedges) with a virtual edge in common are merged. Suppose one
	of these merges $m^\prime$ corresponds to a pair of virtual edges $(e_1,e_2)$
	created by a 2.5-split in $s^\prime$. By definition of a 2.5-split, $e_1$ or $e_2$ is incident to a vertex of degree $2$, say $e_1$. As this
	remains true after any sequence of further splits, we know that $m^\prime$ is a
	merge of two cycles. By Lemma~\ref{lem: noGreenEdgesAfter25}, this implies that
	$e_1$ and $e_2$ are both part of a vertex-edge separator in their respective
	graphs in $\mathcal{I}$. The only graphs in $\mathcal{I}$ containing
	vertex-edge separators are cycles. Thus, in $\mathcal{I}$ there exist two
	cycles containing corresponding virtual edges. A contradiction.
\end{proof}

Let $H$ be a graph and $e \in E(H)$.
Recall that the \emph{ear} of $e$ in $H$  is the maximal (possibly closed) path in $H$ that contains $e$ such that all its internal vertices are of degree two in $H$.
A subgraph~$P$ of $H$ is called an \emph{ear} if it is an ear of some edge of $H$.
If both endvertices of $e$ are of degree at least~3 in $H$, then the ear of $e$ in $H$ is \emph{trivial}, that is, the ear is the length-1 path containing~$e$.
Otherwise it is called \emph{non-trivial}.
\begin{lemma} \label{lem: no2VertexSplitInTriconnectedFrom25Connected}
	Let $s = s_1\dots s_k$ be a sequence of splits of a biconnected graph $G$ such that the resulting graphs are the triconnected components of $G$.
	\begin{enumerate}[(a)]
		\item \label{itm: 2vertexNeverSplit} None of the separation pairs that correspond to the splits in $s$ contains a vertex which is of degree 2 when the split is carried out.
		\item \label{itm: atLeastOneEarIsTrivial} Let $i \in \{1, \dots, k\}$. Consider the graphs $G_1, \dots, G_{i+1}$ obtained from carrying out $s_1, \dots, s_i$ on $G$.
		Let $e, e' \in \bigcup_{j=1}^{i+1} E(G_j)$ be corresponding virtual edges.
		If $e_1$ lies on a non-trivial ear, then $e_2$ lies on a trivial ear.
		\item \label{obs: virtualEars} 
		Let $H_1$ and $H_2$ be triconnected components of $G$ containing corresponding virtual edges $e_1\in E(H_1)$ and $e_2\in E(H_2)$.
		If the ear of $e_1$ in $H_1$ is non-trivial, then $H_1$ is a cycle and the ear of $e_2$ in $H_2$ is trivial.
	\end{enumerate}
\end{lemma}

\begin{proof}
	We prove part~\eqref{itm: 2vertexNeverSplit}.
	Suppose towards a contradiction that $s$ contains a split $s_i$ which splits the graph $H$ at a separation pair $\{u,v\}$ with $\deg_H(u) = 2$.
	It follows by a simple induction on the number of splits succeeding $s_i$ that amongst the triconnected components of $G$ there are two graphs~$G_1$ and~$G_2$ such that $u \in V(G_1) \cap V(G_2)$ and $\deg_{G_1}(u) = \deg_{G_2}(u) = 2$ and, in both graphs $G_1$ and~$G_2$, $u$ is incident to a virtual edge that corresponds to~$s_i$.
	The only triconnected components that contain degree-2 vertices are cycles and, hence, $G_1$ and~$G_2$ are cycles with a common virtual edge.
	That contradicts the construction of the triconnected components, where cycles with a common virtual edge are merged.
	
	Now, part~\eqref{itm: atLeastOneEarIsTrivial} follows with similar considerations.
	Suppose towards a contradiction that two corresponding virtual edges $e$ and $e'$ each lie on a non-trivial ear after carrying out splits $s_1 \dots s_i$.
	In particular, $e$ and~$e'$ are both incident to a degree-2 vertex.
	As above, this implies that both, $e$ and $e'$ are corresponding virtual and contained in distinct cycles of the triconnected components of $G$, which is a contradiction.
	
	Since triconnected components with degree-2 vertices are cycles, part~\eqref{obs: virtualEars} is a consequence of~\eqref{itm: atLeastOneEarIsTrivial}.
\end{proof}

\begin{theorem}[Unique colouring of virtual edges]
	\label{thm: redGreenVirtualEdges}
	Let $s=s_1s_2\dots s_k$ be a sequence of splits that is carried out on a graph $G$ such that the obtained graphs are the triconnected components of $G$.
	We define a 2-colouring of the virtual edges of the triconnected components starting from the uncoloured graph~$G$. For $i \in \{1, \dots, k\}$:
	\begin{itemize}
		\item[\textperiodcentered] If $s_i$ is a non-supporting split, then the respective virtual edges are coloured red.
		\item[\textperiodcentered] If $s_i$ is a 2.5-split for some vertex-edge-separator $(c,e)$ where $e$ is a green virtual edge or a non-virtual edge, then let $e^{\star}$ and $e^{\star\star}$ be the virtual edges arising from $s_i$.
		Colour all virtual edges with labels that appear in the ear of $e^{\star}$ and $e^{\star\star}$ green.
		\item[\textperiodcentered] If $s_i$ is a 2.5-split only for vertex-edge-separators $(c,e)$ with $e$ red, then the virtual edges corresponding to $s_i$ are coloured red.
	\end{itemize}
	The colouring of the virtual edges of the triconnected components obtained this way is independent of the choice of $s$.
\end{theorem}
See Figure~\ref{fig: triconVs25Conn} for an example of the above colouring.

\begin{proof}
	We prove the following more general statement:
	
	\medskip
	\noindent
	\textbf{Claim 1:} \textit{Let $i \in \{0, 1, \dots, l\}$ and let $G_1, G_2, \dots, G_{i+1}$ be the graphs that are obtained from carrying out the splits $s_1, s_2, \dots, s_i$. It holds for each virtual edge $e^{\star} \in \bigcup_{j=1}^{i+1}E(G_j)$ that $e^{\star}$ is coloured green if and only if $e^{\star}$ or its corresponding edge is contained in an ear with at least one non-virtual edge.
		Otherwise it is red.}

	\medskip
	\noindent
	Internal vertices of ears are of degree~2. Thus by Lemma~\ref{lem: no2VertexSplitInTriconnectedFrom25Connected}, internal vertices are never contained in a separation pair that corresponds to one of the splits $s_1, \dots, s_k$, that is,
	\begin{equation}\label{eq: ears never split}
	\text{ears are never split by the sequence~}s_1s_2 \dots s_k.
	\end{equation}
	
	We prove Claim~1 by induction on $i$.
	If $i = 0$, then no split is carried out and, hence, there are no virtual edges to consider and Claim~1 satisfied.
	
	Now let $i \geq 1$.
	By induction, Claim~1 holds for the graphs $G_1', G_2', \dots, G_i'$ obtained from carrying out $s_1, \dots, s_{i-1}$.
	Without loss of generality, $s_i$ splits $G_i'$ into $G_i$ and $G_{i+1}$. Let $e^{\star}_i \in E(G_i)$ and $e^{\star}_{i+1} \in E(G_{i+1})$ be the new virtual edges.
	
	First assume that $s_i$ is a non-supporting split.
	The edges $e_i^{\star}$ and $e_{i+1}^{\star}$ are red and have trivial ears since $s_i$ is non-supporting. Thus, $e_i^{\star}$ and $e_{i+1}^{\star}$ satisfy Claim~1. Other virtual edges and their ears remain unchanged by $s_i$.
	
	Now assume that $s_i$ supports a vertex-edge-separator $(c,e)$ of $G_i'$.
	We may assume that $e \in E(G_i)$. In particular,
	\begin{equation} \label{eq: same ear}
	\text{$e$ and $e^{\star}_i$ lie on the same ear $P$ of $G_i$.}
	\end{equation}
	By Lemma~\ref{lem: no2VertexSplitInTriconnectedFrom25Connected}\eqref{itm: atLeastOneEarIsTrivial} and~\eqref{eq: same ear} it holds that
	\begin{equation}\label{eq: otherEarTrivial}
	\text{all virtual edges that correspond to an edge of $P$ lie on a trivial ear.}
	\end{equation}
	If an ear in $G_i^\prime$ is lengthened by $s_i$, then the ear is a subpath of $P$ according to~\eqref{eq: same ear} and Lemma~\ref{lem: no2VertexSplitInTriconnectedFrom25Connected}\eqref{itm: atLeastOneEarIsTrivial}. In particular, it suffices to prove Claim~1 for the virtual edges of $P$.
	
	If $e$ is non-virtual, 
	then all virtual edges of $P$ and their corresponding edges are coloured green and Claim~1 is satisfied. If~$e$ is green,
	then by induction $e$ or its corresponding edge lie in a non-trivial ear of $G_i'$ which contains a non-virtual edge.
	This ear is a subpath of $P$ by Lemma~\ref{lem: no2VertexSplitInTriconnectedFrom25Connected}\eqref{itm: atLeastOneEarIsTrivial} and, hence, Claim~1 is satisfied.
	If $e$ is red, then the ear $P'$ of $e$ in $G_i'$ solely consists of red edges by induction. If $s_i$ supports a vertex-edge-separator with a green or non-virtual edge, then one of the above cases applies. Otherwise, $P$ is the union of the trivial ears $G_i'[e^{\star}_i]$, $P^\prime$, and possibly one additional ear that contains a red edge of a vertex-edge-separator supported by $s_i$. All of the ears consist solely of virtual red edges. This settles the claim.
\end{proof}

\begin{corollary}
	\label{coro: redVirtualEdges}
	Let $G$ be a biconnected graph and let $e$ and $e'$ be corresponding virtual edges of the triconnected components of $G$.
	Apply the edge-colouring of Theorem~\ref{thm: redGreenVirtualEdges}.
	The following statements are equivalent:
	\begin{itemize}
		\item[\textperiodcentered] $e$ is red.
		\item[\textperiodcentered] $e'$ is red.
		\item[\textperiodcentered] The ears of $e$ and $e'$ in the triconnected components are both trivial, or, one of the two ears is a cycle solely consisting of virtual edges and the other ear is trivial.
	\end{itemize}
\end{corollary}

\begin{proof}
	Let $e$ and $e'$ be two corresponding virtual edges.
	By Claim 1 of the above proof, $e$ is red if and only if the ear of $e$ and the ear of $e'$ solely consist of virtual edges.
	If we consider triconnected components, then
	it follows from Lemma~\ref{lem: no2VertexSplitInTriconnectedFrom25Connected}\eqref{obs: virtualEars} that the two ears are trivial or one of them is a cycle solely consisting of virtual edges.
\end{proof}

In Chapter~\ref{sec: linearTime} we will exploit Corollary~\ref{coro: redVirtualEdges} to develop a linear time algorithm that computes the 2.5-connected components of a given graph.

\begin{theorem}[Uniqueness of 2.5-connected components]
	\label{thm: 2.5 unique}
	If $\mathcal{I}$ and $\mathcal{I}$' are two sets of 2.5-connected components of the same biconnected graph, then $\mathcal{I}$ and $\mathcal{I}'$ are equivalent.
	With respect to the colouring described in Theorem~\ref{thm: redGreenVirtualEdges},
	the 2.5-connected components of $G$ are obtained from the unique triconnected components by merging all red edges of the triconnected components of $G$.
\end{theorem}

\begin{proof}
	We use the same edge colouring as in Theorem~\ref{thm: redGreenVirtualEdges}.
	Let $s$ be a sequence of 2.5-splits that leads to some 2.5-connected components $H_1, H_2, \dots, H_l$ of~$G$.
	Observe that all virtual edges that correspond to the splits in $s$ are green.
	By Lemma~\ref{lem: triconnectedFrom25Connected} there exists a sequence of splits $s'$ such that the triconnected components of $G$ are obtained by carrying out $ss'$.
	As a direct consequence of Lemma~\ref{lem: noGreenEdgesAfter25}, all virtual edges that correspond to splits in $s'$ are red.
	
	Altogether, splits from $s$ correspond to green edges and splits from $s'$ correspond to red edges. 
	However, we know that the red-green colouring of the virtual edges of triconnected components is independent of the choice of~$s$ and~$s'$ by Theorem~\ref{thm: redGreenVirtualEdges}.
	This implies that any sequence that leads to 2.5-connected components corresponds to the same set of virtual edges of the triconnected components. 
	This settles the claim.
\end{proof}

\begin{corollary}\label{coro: 25TreeMinorOfTriTree}
	Let $G$ be a graph.
	If $\mathcal{I}$ denotes the 2.5-connected components of $G$ and $\mathcal{I}'$ denotes the triconnected components of $G$, then $S(\mathcal{I})$ is a minor of $S(\mathcal{I}')$.\footnote{Recall the definition of $S(\mathcal{I})$ from Lemma~\ref{lem: splitsSuffice}.}
\end{corollary}	

\begin{proof}
	This follows from Theorem~\ref{thm: 2.5 unique} since merging two components corresponds to contracting an edge of $S(\mathcal{I}')$.
\end{proof}

\begin{corollary}\label{coro: 2.5ConnectivityTest}
	A biconnected graph is 2.5-connected if and only if no cycle of its triconnected components contains a non-virtual edge.
\end{corollary}

\section{A Linear Time Algorithm for 2.5-Connected Components} \label{sec: linearTime}
Based on the work of Hopcroft and Tarjan~\cite{hopcroftTarjan1973} Gutwenger and Mutzel~\cite{gutwengerMutzel2000} showed that the triconnected components of a given graph can be computed in linear time.
In this section, we provide a linear-time algorithm which computes the 2.5-connected components of a graph given its triconnected components.
It follows that the 2.5-connected components of a given graph can be computed in linear time.
The main idea is again, to exploit the red-green colouring of the virtual edges in order to obtain the 2.5-connected components from the triconnected components.

\begin{theorem}\label{thm: linearAlgo}
	The 2.5-connected components of a biconnected graph can be computed in 
	linear time.
\end{theorem}

\begin{proof}
	Let $G$ be a biconnected graph and denote by $E^\prime$ the set of all
	virtual edges in the triconnected components of $G$. By~Gutwenger and 
	Mutzel~\cite{gutwengerMutzel2000} the triconnected components as well 
	as the set of virtual edges can be computed in linear time. 
	It remains to determine those virtual edges that need to be merged again 
	in order to get the 2.5-connected components of $G$. By 
	Theorem~\ref{thm: 2.5 unique} these are the red edges defined in 
	Theorem~\ref{thm: redGreenVirtualEdges}. By Corollary~\ref{coro: 
		redVirtualEdges} an edge $e$ is coloured red if and only $e$ and its corresponding edge lie on a trivial ear or one of the two ears is a cycle solely consisting of virtual edges.
	Clearly, we can find these virtual edges in linear time by moving through the 
	tree structure given by the triconnected components, taking into 
	account that the number of virtual edges is linear in the number of 
	vertices and edges of $G$, cf.~\cite{gutwengerMutzel2000}. Further 
	each merge can be realised in constant time which gives us the desired 
	result.		
\end{proof}

\section{Critical 2.5-Connected Graphs} \label{sec: CriticalGraps}
In this chapter, we provide novel decomposition techniques for critical 2.5-connected graphs.
In analogy to Tutte's well-known decomposition theorem (Theorem~\ref{thm: tutte}) we show that critical 2.5-connected graphs which are not isomorphic to the $K_4$ can be reduced to critical 2.5-connected graphs of smaller order using simple graph operations.

Let $G$ be a biconnected graph.
A vertex-2-edge-separator of $G$ is a triple $(c,e_1,e_2) \in V(G) \times E(G)^2$ such that $G-e_1-e_2-c$ is disconnected.
A graph~$G$ is \emph{critical 2.5-connected} if $G$ is 2.5-connected and for every edge $e \in E(G)$ it holds that $G-e$ is not 2.5-connected, that is, $e$ is contained in a vertex-2-edge-separator of $G$.
If~$u \in V(G)$ is a degree-3 vertex with incident edges~$e_0$, $e_1$, $e_2$, then the vertex-2-edge-separator~$(c, e_1, e_2)$ is \emph{degenerate}, where~$c$ denotes the vertex that is joined to~$u$ by~$e_0$.
A critical 2.5-connected graph is \emph{degenerate} if every vertex-2-edge-separator is degenerate.
Consider prisms of order at least~8 or complete bipartite graphs isomorphic to $K_{3,n}$ with $n \geq 3$ as examples for infinite families of degenerate graphs.

\begin{theorem}\label{thm: criticalWith2Seps}
	A 2.5-connected graph $G$ with triconnected components $\mathcal{I}$
	is critical if and only if the following conditions are satisfied:
	\begin{enumerate}[(a)]
		\item every $k$-edge $M \in \mathcal{I}$ containing a non-virtual edge is a 3-edge that contains exactly one virtual edge and the unique neighbour of $M$ in $\mathcal{S}(\mathcal{I})$ is a cycle, \label{itm: k edges}
		\item \label{itm: other cpnts} every other component $H \in \mathcal{I}$ satisfies that each non-virtual edge of $H$ lies on a vertex-2-edge-separator of $H$ with both edges non-virtual.
	\end{enumerate}
\end{theorem}
\begin{proof}
	First assume that $G$ is critical 2.5-connected.
	Let $H$ be a triconnected component of $G$ which contains a non-virtual edge $e_1 \in E(H)$.
	Since $G$ is critical, there exists an edge $e_2$ and a vertex $c$ in $G$ such that $(c, e_1, e_2)$ is a vertex-2-edge-separator of $G$.
	
	Suppose towards a contradiction that $e_2 \notin E(H)$.
	Let $H'$ be the triconnected component of $G$ with $e_2 \in E(H')$.
	It follows from Corollary~\ref{coro: 2.5ConnectivityTest} that neither $H$ nor $H'$ is a cycle.
	In particular, $H-e_1$ and $H'-e_2$ are biconnected graphs.
	Now, merging preserves biconnectivity and, hence, $G-e_1-e_2$ is biconnected.
	This is a contradiction since $G-e_1-e_2-c$ is disconnected.
	
	So far, we have shown that $e_1$ and $e_2$ are both contained in the same triconnected component $H$ of $G$ which is not a cycle.
	
	First assume that $H$ is not a $k$-edge.
	If $H \in \mathcal{T}$ and $c \in V(H)$, then $(c,e_1,e_2)$ is a vertex-2-edge-separator of $H$.
	(A virtual edge with ends in distinct components of $H-e_1-e_2-c$ would imply the existence of a path in $G-e_1-e_2-c$ between the components which is a contradiction.)
	Therefore let $c \notin V(H)$. Let $e^{\star}$ be the virtual edge in $H$ with the following property:
	Removing the edge corresponding to the label of $e^{\star}$ from $S(\mathcal{I})$ disconnects the triconnected components containing $c$ from $H$.
	Then $H-e_1-e_2-e^{\star}$ is disconnected. Choose a suitable endvertex~$u$ of $e^{\star}$ to obtain the desired vertex-2-edge-separator of~$H$.
	
	Now, let $H$ be a $k$-edge and denote the number of virtual edges in $H$ by~$i$.
	If $i=0$, then $G = H$ which is a contradiction since $H$ is not critical.
	If $i\in\{3, \dots, k-1\}$, then denote by $e'$ a non-virtual edge in $E(H)$.
	We claim that $G-e'$ is 2.5-connected.
	The triconnected components $\mathcal{I}'$ of $G-e'$ are $\mathcal{I}\cup \{H-e'\} \setminus \{H\}$ and, hence, every cycle in $\mathcal{I}'$ is free of non-virtual edges. By Corollary~\ref{coro: 2.5ConnectivityTest}
	$G-e'$ is 2.5-connected which contradicts the assumption.
	If $i=2$, then denote the two neighbour graphs of $H$ in $S(\mathcal{I})$ by $G_1$ and $G_2$. For $j \in \{1,2\}$ let $e_j \in E(G_j)$ be the edge that corresponds to a virtual edge in~$H$.
	Relabel $e_1$ and $e_2$ such that they are corresponding virtual edges. Denote the resulting graphs by $G_1'$ and $G_2'$.
	The triconnected components of $G-e'$ are given by $\mathcal{I} \cup \{G_1', G_2'\}\setminus \{G_1, G_2, H\}$ and, hence, do not contain cycles solely consisting of virtual edges. Analogously to the above case, this implies that $G-e'$ is 2.5-connected which is a contradiction.
	
	If $i = 1$, then $H$ is a leaf of $S(\mathcal{I})$.
	Since $G$ does not properly contain 3-edges as subgraphs by assumption, we obtain $k = 3$.
	The graph $G'$ obtained by merging all components in $\mathcal{I}\setminus \{H\}$ at corresponding virtual edges is biconnected since being biconnected is preserved under merges.
	Denote by $e_1$ and $e_2$ the two non-virtual edges of $H$.
	If $G-e_1-\hat{e}-c$ is disconnected for some $(c,\hat{e}) \in V(G)\times E(G)$, then $\hat{e} = e_2$. 
	However, $G' = G-e_1-e_2$ is biconnected which is a contradiction.
	
	\medskip
	
	Now assume that conditions~\eqref{itm: k edges} and~\eqref{itm: other cpnts} are satisfied.
	Let $e_1 \in E(G)$ and let $H\in \mathcal{I}$ with $e_1 \in E(H)$.
	If $H$ is not a $k$-edge, then according to~\eqref{itm: other cpnts} there exists $(c,e_2)\in V(H)\times E(H)$ such that $e_2$ is non-virtual and $(c, e_1, e_2)$ is a vertex-2-edge-separator of $H$.
	Then $G-e_1-e_2-c$ is disconnected since each path in $G$ connecting vertices of $H$ with edges outside of $H$ is represented by virtual edges in $H$.
	
	Otherwise, $H$ is a 3-edge with two non-virtual edges $e_1$ and $e_2$ and there exists a cycle $C \in \mathcal{I}$ adjacent to $H$ in $S(\mathcal{I})$.
	Choose a vertex $c \in V(C)\setminus V(H)$.
	Then $(c, e_1, e_2)$ is a vertex-2-edge separator of $G$.
	This settles the claim.
\end{proof}

\begin{theorem} \label{thm: reductionV2eSeparators}
	Let $G$ be a 3-connected graph that contains a non-de\-ge\-ne\-ra\-te ver\-tex-2-edge-separator $(c,e_1, e_2)$.
	\begin{enumerate}[(a)]
		\item \label{itm: 3VertexSeparatorConstr} If $c$ is incident to an edge $e_0 \in E(G)$ such that $G-e_0-e_1-e_2$ is disconnected with components $C_1$ and $C_2$, then let $G_1$ ($G_2$) be the graph constructed by adding a new vertex $x_1$ ($x_2$) and the edges $u_ix_1$ ($v_ix_2$) for $i \in \{0, 1,2\}$, where $u_i$ ($v_i$) denotes the endvertex of $e_i$ in $C_1$ ($C_2$).
		\item \label{itm: v2eSeparatorConstr} Otherwise, there are exactly two components $C_1$ and $C_2$ of $G-e_1-e_2-c$.
		Let $G_1$ ($G_2$) be the graph constructed from $G[V(C_1) \cup \{c\}]$ ($G[V(C_2) \cup \{c\}]$) by adding a new vertex $x_1$ ($x_2$) and the edges $u_1x_1$, $u_2x_1$, and $cx_1$ ($v_1x_2$, $v_2x_2$, and $cx_2$), where $u_i$ ($v_i$) is the endvertex of $e_i$ in $C_1$ ($C_2$).
	\end{enumerate}
	If $G$ is critical 2.5-connected, then $G_1$ and $G_2$ are critical 2.5-connected 3-connected graphs of smaller order than $G$.
\end{theorem}

\begin{proof}
	Assume that the constraints of~\eqref{itm: 3VertexSeparatorConstr} are satisfied.
	We may restrict ourselves to proving that $G_1$ is a critical 2.5-connected 3-connected graph of lesser order than $G$.
	If follows from Menger's theorem that $G_1$ is a 3-connected graph ($\star$).
	This implies that $G_1$ is 2.5-connected.
	It remains to show that $G_1$ is critical 2.5-connected.
	
	Suppose towards a contradiction that $G_1-\hat{e}_1$ is 2.5-connected for some $\hat{e}_1 \in E(G_1)$.
	Since $\{u_1x_1, u_2x_1, cx_1\}$ disconnects $x_1$ from the rest of $G_1$, we know that $\hat{e}_1 \in E(G)$.
	Let $\hat{e}_2 \in E(G)$ and $\hat{c} \in V(G)$.
	If $\hat{e}_2 \in \{e_0, e_1, e_2\}$, then $G-\hat{e}_1-\hat{e}_2-\hat{c}$ is connected as a consequence of the 3-connectivity of $G_1$ and~$G_2$.
	From now on, we assume that $\hat{e}_2 \notin \{e_0, e_1, e_2\}$.
	If $\hat{c} \in V(G_1)\setminus\{x_1\}$, then $G_1-\hat{e}_1-\hat{e}_2-\hat{c}$ is connected by the assumption on $\hat{e}_1$ and $G_2-\hat{e}_2 - x_2$ is connected by~$(\star)$.
	Then $G_1-\hat{e}_1-\hat{e}_2-\hat{c} - x_1$ has at most three components, each containing at least one vertex from $\{u_0, u_1, u_2\}$.
	This implies that $G - \hat{e}_1-\hat{e}_2-\hat{c} = G_1 - x_1 \cup G_2 - x_2 - \hat{e}_1-\hat{e}_2-\hat{c} + e_1 + e_2 + e_3$ is connected for any choice of $\hat{e}_2$ and $\hat{c}$ which is a contradiction.
	If, otherwise $\hat{c} \in V(G_2)\setminus\{x_2\}$, then $G_1-\hat{e}_1-\hat{e}_2$ is biconnected and $G_2-\hat{e}_2 - x_2 - \hat{c}$ has at most three components, each containing at least one vertex from $\{v_0, v_1, v_2\}$. As above, we obtain a contradiction to $G$ being critical since $G - \hat{e}_1-\hat{e}_2-\hat{c}$ is connected for any choice of $\hat{e}_2$ and $\hat{c}$.
	
	\medskip
	\noindent In order to prove~\eqref{itm: v2eSeparatorConstr}, observe that it follows from the 3-connectivity of $G$ that there are exactly two components $C_1$ and $C_2$ of $G-e_1-e_2-c$. Now, we may apply the exact same arguments as in~\eqref{itm: 3VertexSeparatorConstr}.
	This settles the claim.
\end{proof}

The only 3-connected critical graphs which cannot be decomposed into critical graphs of smaller order using operations above are degenerate graphs.

\begin{theorem}\label{thm: reductionOfDegeneratedGraphs}
	Let $G$ be a degenerate 3-connected graph which is not 3-regular and let $u \in V(G)$ with $\deg_G(u) = 3$. Denote the neighbours of $u$ by $v_1, v_2,$ and~$v_3$.
	\begin{enumerate}[(a)]
		\item \label{itm: 3vertex with geq4nbs} If $\deg_G(v_i) \geq 4$ for $i \in \{1,2,3\}$, then set $G' \coloneqq G-u$.
		\item \label{itm: 3vertex with 1 3nb and 1 geq4nb} If $\deg_G(v_1) = 3$ and $\deg_G(v_3) \geq 4$, then set $G' \coloneqq G-u + v_1v_2$.
	\end{enumerate}
	The graph $G'$ is critical 2.5-connected with $|V(G')| < |V(G)|$.
\end{theorem}
\begin{proof}
	Let~$H$ be a 2.5-connected graph. Observe that
	\begin{equation} \label{eq: allIncidentTo3VertexImpliesCritical}
	\text{if every edge in $H$ is incident to a degree-3 vertex, then $H$ is critical.}
	\end{equation}
	This follows since an edge that is incident to a degree-3 vertex lies in a degenerate separator.
	Vice versa, it holds that
	\begin{equation} \label{eq: degenerateImpliesAllIncidentTo3Vertex}
	\text{if $H$ is degenerate, then every edge of $H$ is incident to a degree-3 vertex.}
	\end{equation}
	Suppose that a degenerate graph $H$  contains a triangle.
	It follows from~\eqref{eq: degenerateImpliesAllIncidentTo3Vertex} that at least two vertices, say $u$ and $v$, of the triangle are of degree~3. Let $w$ denote the third vertex of the triangle and denote by $e_u$ ($e_v$) the unique edge that is incident to $u$ ($v$) but is not contained in the triangle. Now $(w, e_u, e_v)$ is a non-degenerate separator if $H$ is not a complete graph on four vertices. This is a contradiction. We obtain that
	\begin{equation}\label{eq: triangle free}
	\text{if a degenerate graph is not isomorphic to the $K_4$, then it is triangle-free.}
	\end{equation}

	Every edge in $G$ is incident to a degree-3 vertex by~\eqref{eq: degenerateImpliesAllIncidentTo3Vertex}.
	This is maintained when we construct $G'$.
	It follows from~\eqref{eq: allIncidentTo3VertexImpliesCritical} that it suffices to prove that $G'$ is 2.5-connected.
	Suppose towards a contradiction that $G'$ contains a vertex-edge-separator $(c,e)$.
	If the neighbourhood of $u$ is contained in one component of $G'-c-e$, then $(c,e)$ is a vertex-edge-separator of $G$ which is a contradiction.
	Therefore, $v_1, v_2,$ and $v_3$ are not all in the same component of $G'-c-e$.
	
	First assume~\eqref{itm: 3vertex with geq4nbs}. Without loss of generality $v_1$ is in a different component of $G'-c-e$  than $v_2$ and $v_3$. Consequently $(c, e, uv_1)$ is a non-degenerate separator in the degenerate graph~$G$ which is a contradiction.
	
	Now assume~\eqref{itm: 3vertex with 1 3nb and 1 geq4nb}.
	Observe that $\{v_1, v_2, v_3\}$ is an independent set in $G$ by~\eqref{eq: triangle free}.
	If $v_1$ is in a different component than $v_2$ and $v_3$ in $G'-c-e$, then $e= v_1v_2$ or $c = v_2$. In the first case $G-c-uv_1$ is disconnected which contradicts that~$G$ is 2.5-connected. In the second case $(c, e, uv_1)$ is a non-degenerate separator of~$G$ which contradicts the degeneracy of~$G$.
	Interchanging the roles of~$v_1$ and~$v_2$ leads to a contradiction if $v_2$ is separated from  $v_1$ and $v_3$ by $(c,e)$.
	
	Last assume that $v_3$ does not share a component with $v_1$ and $v_2$ in $G'-c-e$.
	Now $(c,e, uv_3)$ is a non-degenerate separator of $G$ since $\deg_G(v_3) \geq 4$.
	This settles the claim.
\end{proof}

We have shown in this chapter that critical 2.5-connected graphs can be reduced using simple operations until the obtained graphs are 3-regular and 3-connected.
Then we may apply the following theorem of Tutte.
\begin{theorem}[\cite{wormald1979}, cf.~\cite{tutte1966}] \label{thm: tutte}
	Each simple 3-connected 3-regular graph other than a complete graph on four vertices can be obtained from a 3-connected 3-regular graph $H$ by subdividing two distinct edges of $H$ and connecting the subdivision vertices with a new edge.
	Conversely, each graph obtainable in this way is 3-connected.
\end{theorem}

The graph $H$ in Theorem~\ref{thm: tutte} is 3-regular and 3-connected and, hence, $H$ is critical 2.5-connected.
\begin{figure}[h]				
	\centering		
	\begin{tikzpicture}[scale=0.43]
	\begin{scope}[shift={(-3,0)}, scale=.5]
	\draw 	
	(2,-5) node (1) [svertex] {}
	(3.3,-5.2) node (l1) [novertex] {$v_3$}
	(-2,-5) node (2) [svertex] {}
	(-5,-2) node (3) [svertex] {}
	(-5,2) node (4) [svertex] {}
	(-2,5) node (5) [svertex] {}
	(-3,5.5) node (l5) [novertex] {$v_2$}
	(2,5) node (6) [svertex, myRed] {}
	(3,5.5) node (l6) [novertex, myRed] {$u$}
	(5,2) node (7) [svertex] {}
	(6.2,2) node (l7) [novertex] {$v_1$}
	(5,-2) node (8) [svertex] {};					
	\draw[thick]
	(1) -- (2) -- (3) -- (4) -- (5) -- (6) -- (7) -- (8) -- (1)  -- (4)
	(1) -- (6)
	(2) -- (5)
	(5) -- (8)
	(3) -- (7);
	\end{scope}	
	
	\begin{scope}[shift={(6,0)}, scale=.5,]
	\draw 	
	(2,-5) node (1) [svertex] {}
	(-2,-5) node (2) [svertex] {}
	(-5,-2) node (3) [svertex] {}
	(-5,2) node (4) [svertex] {}
	(-2,5) node (5) [svertex, red] {}
	(-3,5.5) node (l5) [novertex, myRed] {$c$}
	(5,2) node (7) [svertex] {}
	(5,-2) node (8) [svertex] {}
	(-.2,0.8) node (le1) [novertex, myRed] {$e_1$}
	(2.8,-2.7) node (le2) [novertex, myRed] {$e_2$};					
	\draw[thick]
	(1) -- (2) -- (3) -- (4) -- (5)
	(4) -- (1)
	(2) -- (5) -- (7) -- (8) -- (5);
	\draw[thick, myRed]
	(3) -- (7)
	(8) -- (1);
	\end{scope}
	
	\begin{scope}[shift={(15,0)}, scale=.5,]
	\draw 	
	(2,-5) node (1) [svertex] {}
	(-2,-5) node (2) [svertex] {}
	(-5,-2) node (3) [svertex] {}
	(-5,2) node (4) [svertex] {}
	(-2,5) node (5) [svertex] {}
	(1,0) node (x1) [svertex] {}
	(1.9,.8) node (lx1) [novertex] {$x_1$};					
	\draw[thick]
	(4) -- (1) -- (2) -- (3) -- (4) -- (5) -- (2)
	;
	\draw[thick]
	(5) -- (x1) -- (1)
	(x1) -- (3);
	\end{scope}
	
	\begin{scope}[shift={(16.5,0)}, scale=.5,]
	\draw 	
	(-2,5) node (5) [svertex] {}
	(5,2) node (7) [svertex] {}
	(5,-2) node (8) [svertex] {}
	(1,0) node (x2) [svertex] {}
	(1.1,-1.1) node (lx2) [novertex] {$x_2$};
	\draw[thick]
	(5) -- (7) -- (8) -- (5);
	\draw[thick]
	(5) -- (x2) -- (7)
	(x2) -- (8);
	\end{scope}
	
	\end{tikzpicture}
	\label{fig: reductionExample}	
	\caption{Reduction of a degenerate graph.}	
\end{figure}
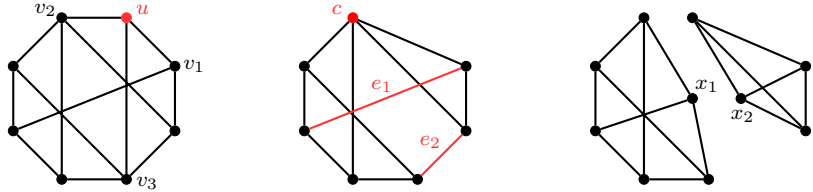
We close this chapter with an example.
Consider Figure~\ref{fig: reductionExample}.
The left graph is degenerate and not 3-regular.
We apply Theorem~\ref{thm: reductionOfDegeneratedGraphs}\eqref{itm: 3vertex with 1 3nb and 1 geq4nb} to obtain the graph in the middle.
This graph is critical 2.5-connected and contains a non-degenerate separator $(c, e_1, e_2)$.
We obtain the isomorphic copies of the $K_{3,3}$ and the $K_4$ on the right by carrying out the construction of Theorem~\ref{thm: reductionV2eSeparators}\eqref{itm: v2eSeparatorConstr}.
Observe that both of the graphs on the right are 3-regular and 3-connected.
We may now apply Theorem~\ref{thm: tutte} to reduce the bipartite graph further while the critical 2.5-connectivity is preserved.

\section{Application to Extremal Cycle Decomposition} \label{sec: Applications}
In this section, we prove that the problem of finding an extremal cycle decomposition of an Eulerian graph can be reduced to finding an extremal cycle decomposition for its 2.5-connected components.
Furthermore, we show how Hajós' conjecture can be reduced to considering 2.5-connected components.
A \emph{decomposition} of a graph~$G$ is a set of subgraphs~$\mathcal{C}$ of~$G$ such that each edge of~$G$ is contained in exactly one of the subgraphs.
We say that $G$ can be \emph{decomposed} into the elements of $\mathcal{C}$.
If all of the subgraphs in $\mathcal{C}$ are cycles, then $\mathcal{C}$ is a \emph{cycle decomposition}.
For an Eulerian graph~$G$ we set
\begin{align*}
c(G) &\coloneqq \min \{k \colon G~\text{can be decomposed into $k$ cycles}\}~\text{and}\\
\nu(G) &\coloneqq \max \{k \colon G~\text{can be decomposed into $k$ cycles}\}.
\end{align*}
A cycle decomposition of $G$ with $c(G)$ ($\nu(G)$) cycles is \emph{minimal} (\emph{maximal}).
Let~$G_1$ and $G_2$ be obtained from carrying out a 2.5-split on $G$.
It is proven in~\cite{heinrichStreicher2019} that $c(G) = c(G_1) + c(G_2) - 1$ and $\nu(G) = \nu(G_1) + \nu(G_2)-1$. The theorem below follows.
\begin{theorem}
	Let $G$ be a biconnected Eulerian graph and $G_1, G_2, \dots, G_k$ its 2.5-connected components.
	\begin{enumerate}[(a)]
		\item $c(G) = \sum_{i=1}^k c(G_i) -k + 1$,
		\item $\nu(G) = \sum_{i=1}^k \nu(G_i) -k + 1$.
	\end{enumerate}
\end{theorem}

Hajós' conjecture asserts that an Eulerian graph can be decomposed into at most $\nicefrac{1}{2}(|V(G)| + m(G) - 1)$ cycles, where $m(G)$ denotes the minimal number of edges that need to be removed from $G$ in order to obtain a simple graph.\footnote{Originally, Hajós conjectured that at most $\nicefrac{1}{2}|V(G)|$ cycles are needed. This equivalent reformulation is due to Fan and Xu, cf.~\cite{fan2002}.}
The only progress made towards a verification of Hajós' conjecture concerns graphs that contain vertices of degree at most~4 (cf.~\cite{fan2002}), very sparse graphs (cf.~\cite{fuchs2019}) and, very dense graphs (cf.~\cite{girao2019}).

\begin{theorem}\label{thm: hajos}
	Let $G$ be a biconnected graph.
	If all 2.5-connected components of~$G$ satisfy Hajós' conjecture, then $G$ satisfies Hajós' conjecture.
	
	In particular, the conjecture of Hajós' is satisfied if and only if all 2.5-con\-nec\-ted graphs satisfy Hajós' conjecture.
\end{theorem}

\begin{proof}
	Assume that all 2.5-connected graphs satisfy Hajós' conjecture.
	Let $G$ be an Eulerian graph.
	Granville and Moisiades~\cite{granville1987} proved that it suffices to verify Hajós' conjecture for all biconnected graphs in order to show that all graphs satisfy the conjecture.
	In particular, we may assume that~$G$ is biconnected.
	We prove the following claim:
	\emph{Let~$G_1$ and $G_2$ be obtained from carrying out a 2.5-split on $G$. If $G_1$ and $G_2$ satisfy Hajós' conjecture, then~$G$ satisfies Hajós' conjecture.}
	
	We have $V(G_1) + V(G_2)  = V(G) + 2$ and $m(G_1) + m(G_2) \leq m(G) + 1$.
	Consequently,
	\begin{align*}
	c(G) &= c(G_1) + c(G_2) - 1\\
	& \leq \nicefrac{1}{2}\left(|V(G_1)| + |V(G_2)| + m(G_1) + m(G_2) - 2 \right)-1\\
	&\leq\nicefrac{1}{2}\left(|V(G)| + 2 + m(G) + 1 - 2 \right)\\
	&=\nicefrac{1}{2}\left(|V(G)| + m(G) - 1 \right).
	\end{align*}
	Now, the statement follows by induction on the number of 2.5-connected components of $G$.
\end{proof}

\section{Conclusion} \label{sec: ConclusionAndFurtherResearch}
We provide a canonical decomposition of a biconnected graph into its unique 2.5-connected components.
Furthermore, we show how these components can be constructed from the triconnected components of the graph.
This overall gives a linear-time algorithm for the 2.5-connected components.
We show that all critical 2.5-connected except complete graphs on four vertices can be reduced to smaller critical 2.5-connected graphs.
Finally, we prove that it suffices to verify Hajós' conjecture for all 2.5-connected graphs in order to verify the conjecture for all graphs.

\pagebreak

\bibliography{spqr}
\bibliographystyle{alpha}

\vfill

\pagebreak
\small
\vskip2mm plus 1fill
\noindent
Version \today{}
\bigbreak

\noindent
Irene Heinrich \\
Algorithms and Complexity Group, Department of Computer Science\\
Technische Universit\"at Kaiserslautern, Kaiserslautern\\
Germany\\
ORCiD: 0000-0001-9191-1712\\

\noindent
Till Heller\\
Department of Optimization\\
Fraunhofer ITWM, Kaiserslautern\\
Germany\\
ORCiD: 0000-0002-8227-9353\\

\noindent
Eva Schmidt\\
Optimization Research Group, Department of Mathematics\\
Technische Universit\"at Kaiserslautern, Kaiserslautern\\
Germany\\
ORCiD: 0000-0002-5074-6199\\

\noindent
Manuel Streicher\\
Optimization Research Group, Department of Mathematics\\
Technische Universit\"at Kaiserslautern, Kaiserslautern\\
Germany\\
ORCiD: 0000-0001-5605-7637\\

\end{document}